\documentclass[pdflatex,sn-mathphys]{sn-jnl}% Math and Physical Sciences 

\usepackage{amssymb}
\usepackage{amsthm}
\usepackage{setspace}
\usepackage{caption}
\usepackage{subcaption}
\usepackage{graphicx}
\usepackage{MnSymbol}
\usepackage{cutwin}
\usepackage{lipsum}
\usepackage{stmaryrd}
\usepackage[algo2e]{algorithm2e} 

\usepackage{xcolor, soul}
\sethlcolor{red}

\theoremstyle{thmstyleone}%
\newtheorem{theorem}{Theorem}%  meant for continuous numbers
\newtheorem{lemma}[theorem]{Lemma}
\theoremstyle{thmstyletwo}%
\theoremstyle{thmstylethree}%
\newtheorem{definition}{Definition}%

\begin{document}

\title[Article Title]{Reduction of a 
biochemical network mathematical model by means of approximating activators and inhibitors as perfect inverse relationships.}

\author*[1]{\fnm{Chathranee} \sur{Jayathilaka}}\email{chathranee.aththanapolaarachchilage@monash.edu}

\author[3]{\fnm{Robyn} \sur{Araujo}}\email{r.araujo@qut.edu.au}
\equalcont{These authors contributed equally to this work.}

\author[2]{\fnm{Lan} \sur{Nguyen}}\email{Lan.K.Nguyen@monash.edu}
\equalcont{These authors contributed equally to this work.}

\author*[1]{\fnm{Mark} \sur{Flegg}}\email{Mark.Flegg@monash.edu}

\affil*[1]{\orgdiv{Department of Mathematics}, \orgname{Monash University}, \orgaddress{\street{Clayton}, \state{Victoria}, \country{Australia}}}

\affil[2]{\orgdiv{Monash Biomedicine Discovery Institute}, \orgname{Monash University}, \orgaddress{\street{Clayton}, \state{Victoria}, \country{Australia}}}

\affil[3]{\orgdiv{School of Mathematical Sciences}, \orgname{Queensland University of Technology}, \orgaddress{\street{ Brisbane}, \state{QLD}, \country{Australia}}} %, \city{City}, \postcode{610101},

%%==================================%%
%% abstract %%
%%==================================%%

\abstract
{ Models of biochemical networks are usually presented as connected graphs where vertices indicate proteins and edges are drawn to indicate activation or inhibition relationships. These diagrams are useful for drawing qualitative conclusions from the identification of topological features, for example positive and negative feedback loops. These topological features are usually identified under the presumption that activation and inhibition are inverse relationships. The conclusions are often drawn without quantitative analysis, instead relying on rules of thumb. We investigate the extent to which a model needs to prescribe inhibition and activation as true inverses before models behave idiosyncratically; quantitatively dissimilar to networks with similar typologies formed by swapping inhibitors as the inverse of activators. The purpose of the study is to determine under what circumstances rudimentary qualitative assessment of network structure can provide reliable conclusions as to the quantitative behaviour of the network and how appropriate is it to treat activator and inhibitor relationships as opposite in nature. }


\maketitle

% ------------------------------------------------------------------------
%%% INTRODUCTION
% ------------------------------------------------------------------------

\section{Introduction}\label{sec1}

Whilst there are different approaches to mathematically model biochemical networks, a commonly used model formalism describes continuous changes in concentration of molecular species over time using ordinary differential equations (ODEs), commonly known as reaction rate equations (RREs) \cite{jiang2022identification}. ODE models are studied to better understand generic biochemical network topologies and motifs as well as dynamics of well-defined pathways in specific organisms \cite{vaseghi2001signal, swameye2003identification}. RREs are almost always non-linear as they describe protein interactions \cite{rao2014model}. Reaction rates can vary greatly, giving rise to different time scales. Elaborate integrated interactions give rise to complex nonlinear behaviours, such as ultra-sensitivity~\cite{angeli2004detection}, bi-stability~\cite{rombouts2021dynamic}, and oscillation~\cite{russo2009equilibria}.  
 Numerous studies of RREs have revealed general network properties that give rise to these nonlinear behaviours and, at the same time, give high level understanding for how biochemical networks encode cellular scale responses to stimuli. The general network properties are qualitative in nature, identified by investigating only the network topology, and are now often considered without substantial accompanying quantitative mathematical analysis. Some of these network-scale properties that can be qualitatively associated with complex behaviours include feedback loops, feed-forward loops, cross-talk, compartmentalisation, and noise~\cite{schwartz1999interactions}. Whilst it may be easy to identify a macro-feature responsible for certain behaviour, it is not always easy to \textit{predict} the behaviour simply because the macro-feature is present without quantitative analysis.

 Analysing the topological structure of the network through its interactive diagram is often a focal point for understanding the system’s driving mechanisms and qualitative behaviour \cite{navlakha2014topological, glass1975classification, glass1972co}. A common qualitative technique involves systematically processing the causal links described by the diagram’s edges \cite{gilbert2006petri}. For example, if a node $i$ inhibits node $j$ which inhibits node $k$, it is qualitatively assumed that the two inhibitions counteract each other in such a way that this may be seen as node $i$ activates a `reverse' of node $j$ which then activates node $k$; in both cases, node $i$ indirectly activates node $k$. This qualitative equivalency is shown in Figure \ref{fig:similarity}. Typically, this equivalency is assumed irrespective of the specific mathematical representation of the model. However, the similarities between these two interpretations can vary depending on the quantitative description of activators and inhibitors; in particular, how closely these relationships are opposites/inverses of each other. The ability to discern when two pathways will result in `similar' behaviour is highly useful. In this simple example, we may be able to scientifically agree on a single oriented model (say the one with only activators), which characterises all ``qualitatively similar'' networks. In the absence of any standard form, we define the oriented model as a model which contains all activation interactions along predefined linear pathways embedded within the network. In the case where activators and inhibitors can be interchanged (exactly or approximately), the oriented model characterises a broad collection of models whereby activation and inhibition are able to be used interchangeably as opposites; thereby making it easier to identify similar networks as well as the nature of features such as feedback, feed forward and cross-talk.

 The pertinent question is `What properties allow for the exact or approximate reduction of a biochemical mathematical model to its oriented model?'. We shall focus our attention predominantly on a property of a biochemical model which we call `model bias' (the tendency for an interaction to be conferred directly effecting activation or deactivation of the downstream node) and secondly on a property which we call `model symmetry' (the tendency for an interaction to be conferred directly as a result of elevated or lowered amounts of the upstream node). We find that for `symmetric models' (where equal changes above and below a rest state in the upstream node confer equal and opposite changes in the downstream node), model bias plays a central role in determining the approximate tractability of a network to an oriented form. We investigate how the network connections together with model bias determine the suitability of this oriented form simplification through a numerical study.

The manuscript is structured in the follow way. In Section \ref{sec:Frame}, we mathematically define the framework for orientation of a biochemical network. We also define an unbiased symmetric model; a standard for which exact equivalence between a model and its oriented form can be shown. In Section \ref{sec:Numerics}, we explore increases in intrinsic bias in a model and how this mainly into divergence in its oriented form and in particular how this divergence is regulated specifically by network/pathway properties. We study systematically increasing complexity from simple chain-like pathways, to an example of an integrated pathway with cross-talk. Our studies provide new insights into when reduction of a network with an oriented form \textit{may} be appropriate and when caution should be taken.

\begin{figure}
    \centering
    \includegraphics[width=0.3\textwidth]{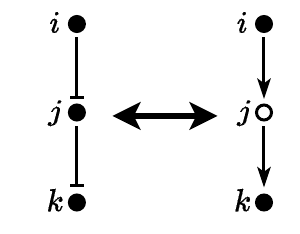}
    \caption{A model diagram of a simple pathway (left) and its oriented form (right). Solid nodes represent protein concentrations and hollow nodes represent a reverse/flip of a node (decreasing if the protein level is increased and vice versa). The oriented form shows identical behaviour in the output node $k$ as the simple pathway under unbiased and symmetric conditions defined later the manuscript. The oriented form allows for the more simple identification that input node $i$ indirectly activates output node $k$. }
    \label{fig:similarity}
\end{figure}

% ------------------------------------------------------------------------
%%% Model FRAMEWORK
% ------------------------------------------------------------------------
\section{Model Framework}
\label{sec:Frame}

In the context of systems biology, the terms `pathway' and `network' are sometimes used interchangeably. A `pathway' typically refers to a relatively small, well-defined set of entities and relationships, such as a specific signal transduction pathway \cite{alberts2002general}. As the name implies, it should be clear in a pathway where the signal `begins' and where it `ends'. On the other hand, a `network' frequently refers to a larger, less constrained set of entities and relationships \cite{azeloglu2015signaling}. A network, on the other hand, is often a set of pathways connected through cross-talk and feedback mechanisms. We will define both a pathway and a network (the latter consisting as a collection of the former) as a directed graph where each node $x_i$, $i=1,\ldots,N$, is represented by a scalar state variable (with the same label), and each edge determines distinct terms on the RHS of the dynamical system describing the rates of change in these variables. Importantly, we use the term `state' rather than concentration. Non-dimensionalisation of the active protein concentrations associated with each node followed by shifting allows us define each state variable on the range $x_i\in[-1,1]$. Here 0 represents half of the saturation concentration and -1 and 1 represents no protein and maximum protein respectively. We choose this framework because we consider that absence (and not just presence) of a protein can cause a response in downstream nodes.

The general mathematical model we investigate consists of the differential equations (\ref{rreswing});
 \begin{equation}\label{rreswing}
\frac{\mathrm{d} \mathbf{x}}{ \mathrm{d} t} = \boldsymbol{\psi}(\mathbf{x}),
\end{equation}
describing the evolution of all $N$ state variables $\mathbf{x} = (x_1,x_2,\ldots, x_N)^{\dagger}$. Here $\psi_i = \varrho_i^+ - \varrho_i^-$ and
\begin{align}
\varrho_i^+ &= \sum_{\mathcal{E}\in \mathcal{E}_i} r_{\mathcal{E}}^+(x_i,y_{\mathcal{E}};\mathcal{T}_{\mathcal{E}}) , \quad \text{and} \label{rplus} \\
\varrho_i^- &= \sum_{\mathcal{E}\in \mathcal{E}_i} r_{\mathcal{E}}^-(x_i,y_{\mathcal{E}};\mathcal{T}_{\mathcal{E}}) , \label{rminus}
\end{align}
where each of the sums are taken over the set of edges $\mathcal{E}_i$ that point towards the node $x_i$, $r_{\mathcal{E}}^\pm$ are functions that completely depend on the model and the type of edge $\mathcal{T}_{\mathcal{E}}\in\{ \mathcal{A},\mathcal{I}\}$ (activator $\mathcal{A}$ or inhibitor $\mathcal{I}$). These are functions of the node $y_{\mathcal{E}}$ from which each edge $\mathcal{E}$ originates. Finally, some edges in a network may not come explicitly from another node in the network. These edges represent sources or external stimuli to the system and must point towards a node in the network and for these edges it is assumed in the context of (\ref{rplus}) and (\ref{rminus}) that $y_{\mathcal{E}} = 0$ such that $y_{\mathcal{E}}$ and $y_{\mathcal{E}}^*$ are non-zero and balanced. We use the superscript asterisk to represent a `flip' in a state variable; negating it algebraically and representing it in diagrams by changing nodes between solid and hollow style.

An edge $\mathcal{E}$ is associated with the qualitative description of an activator or inhibitor. Activation is achieved in one of two ways, either $r_{\mathcal{E}}^+$ relatively increases with $y_{\mathcal{E}}$ and/or $r_{\mathcal{E}}^-$ relatively decreases with $y_{\mathcal{E}}$.  Inhibition is associated with opposite conditions.

In particular, for a given activation edge $\mathcal{E}$ connecting node $y$ to $x$, we require by definition that \begin{equation} \label{condact}\frac{\partial  r_{\mathcal{E}}^+(x,y;\mathcal{A})}{\partial y} \geq \frac{\partial  r_{\mathcal{E}}^-(x,y;\mathcal{A})}{\partial y} \end{equation} everywhere in the state space determined by $x$ and $y$. On the other hand, we require \begin{equation} \label{condin}\frac{\partial  r_{\mathcal{E}}^+(x,y;\mathcal{I})}{\partial y} \leq \frac{\partial  r_{\mathcal{E}}^-(x,y;\mathcal{I})}{\partial y} \end{equation} everywhere if an edge is to be an inhibition edge.

\subsection{Model Bias and Symmetry}

The conditions (\ref{condact}) and (\ref{condin}) can be achieved by adjusting either side of the inequalities. Treating activation as the opposite of inhibition requires that condition (\ref{condin}) is the same as (\ref{condin}) after swapping $+$ and $-$ then swapping the direction of the inequality. We will show it is necessary and sufficient for activation and inhibition to equal and opposite in this respect if both sides of the inequalities are equal and opposite in magnitude (unbiased) and this magnitude is the same for each condition (symmetric). We define these properties mathematically as follows.

\begin{definition} \label{biasedge}
An edge $\mathcal{E}$ of type $\mathcal{T}\in\{ \mathcal{A},\mathcal{I}\}$ is \textit{unbiased} if it defines the function pair $\{ r_\mathcal{E}^+,r_\mathcal{E}^-\}$ in the model (\ref{rreswing}) and (\ref{rplus})-(\ref{rminus}) such that, for all $(x,y)\in [-1,1]\times[-1,1]$,
\begin{equation}
r_{\mathcal{E}}^{+}(x,y;\mathcal{T}) = r_{\mathcal{E}}^{-}(-x,-y;\mathcal{T}).
    \end{equation}
    The edge is considered to be \textit{positively biased} if $r_{\mathcal{E}}^{+}(x,y;\mathcal{T}) > r_{\mathcal{E}}^{-}(-x,-y;\mathcal{T})$ and \textit{negatively biased} if $r_{\mathcal{E}}^{+}(x,y;\mathcal{T}) < r_{\mathcal{E}}^{-}(-x,-y;\mathcal{T})$ for all $(x,y)\in [-1,1]\times[-1,1]$. 
\end{definition}

\begin{definition} \label{biasedmodel}
A model or biochemical network is \textit{positively} or \textit{negatively} biased or \textit{unbiased} if all edges in the model or network are positively or negatively biased or unbiased, respectively. 
\end{definition}

\begin{definition}\label{symedge}
An edge $\mathcal{E}$ is \textit{symmetric} if, in the context of the model (\ref{rreswing}) and (\ref{rplus})-(\ref{rminus}), the functions $r_\mathcal{E}^\pm$ for activation and inhibition are defined such that
\begin{equation}
r_{\mathcal{E}}^{\pm}(x,y;\mathcal{A}) = r_{\mathcal{E}}^{\pm}(x,-y;\mathcal{I}),
    \end{equation}
     for all $(x,y)\in [-1,1]\times[-1,1]$. 
     The edge is considered to be \textit{activator weighted} if $r_{\mathcal{E}}^{\pm}(x,y;\mathcal{A}) > r_{\mathcal{E}}^{\pm}(x,-y;\mathcal{I})$ and \textit{inhibitor weighted} if $r_{\mathcal{E}}^{\pm}(x,y;\mathcal{A}) < r_{\mathcal{E}}^{\pm}(x,-y;\mathcal{I})$ for all $(x,y)\in [-1,1]\times[-1,1]$. 
\end{definition}

\begin{definition} \label{symmodel}
A model or biochemical network is \textit{activator} or \textit{inhibitor} weighted or \textit{symmetric} if all edges in the model or network are {activator} or {inhibitor} weighted or {symmetric}, respectively. 
\end{definition}

In practise, each node of a biochemical network graph/diagram often represents a protein concentration. Concentrations can either be increased/activated (increasing the state variable for the node) or decreased/deactivated (decreasing the state variable for the node). To better explain the mechanisms which determine bias and asymmetry (weighting) in a model, it is useful to add detail to each node and  consider an `active' (solid dots) and `inactive' (hollow dots) component. Of course, these are just labels as we assume that information may equally use the active and inactive components and that these labels just determine the positive and negative direction of the node state variable. A protein $x_i$ can be activated or deactivated and the mathematical model encodes this in the terms $\varrho_i^+$ and $\varrho_i^-$ respectively. Showing generic nodes $x$ and $y$ where $y$ affects $x$ through an edge, it is useful to visualise the mechanistic manner with which bias and weight/asymmetry are manifested in models. We catalogue the different extreme cases for both activator and inhibitor in Figure \ref{fig:definitions}.

\begin{figure}[H]
    \centering
    \includegraphics[width=1\textwidth]{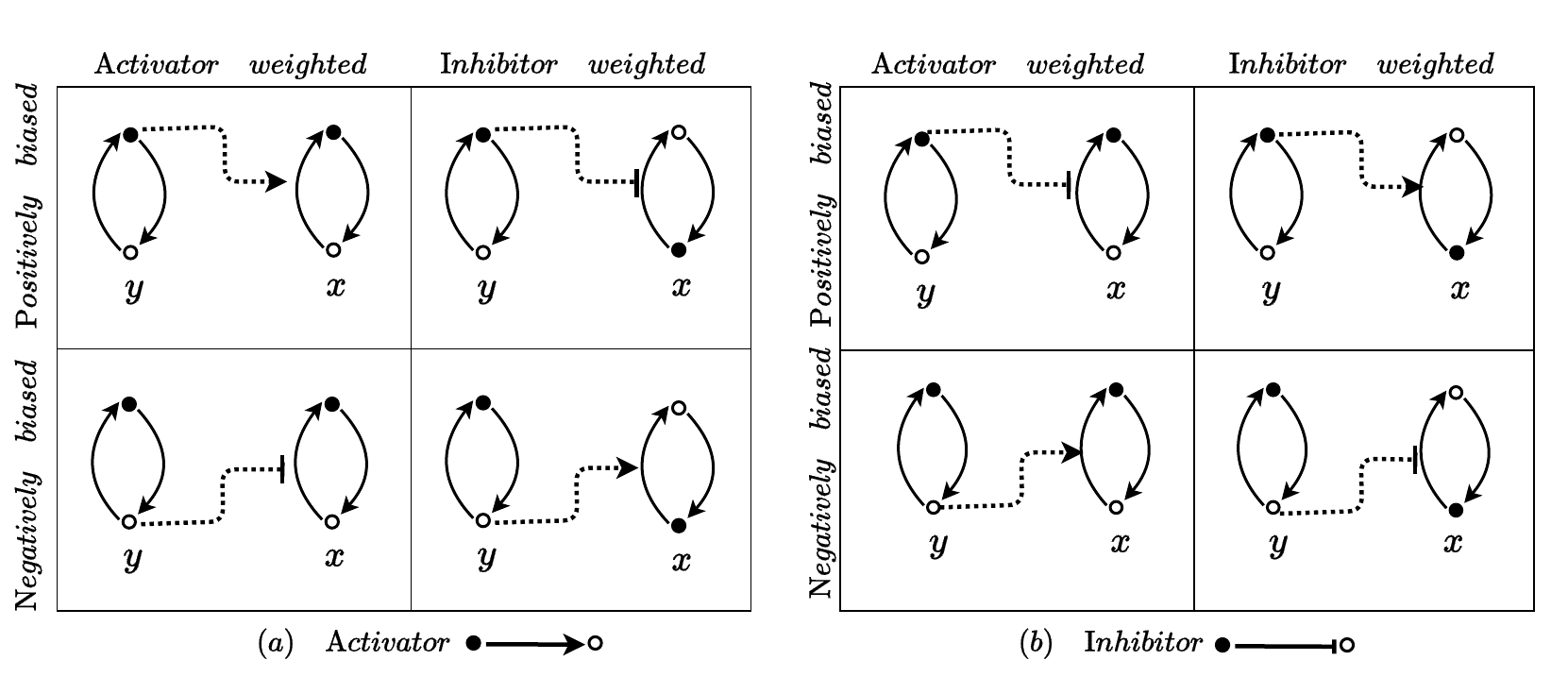}
    \caption{Mechanistic diagrams for asymmetric and biased activator (a) and inhibitor (b) model edges. Each node $y$ (upstream) and $x$ (downstream) is represented as an actively switching chemical species between `active' (solid dot) and `inactive' (hollow dot) forms. The switching from active to inactive form is represented in the model for each node by $\varrho^-$ (Equation (\ref{rminus})) and from inactive to active by $\varrho^+$ (Equation (\ref{rplus})) which are influenced by the network edges. The dotted connectors indicate how biased (negative or positive) and weighted (activator or inhibitor) models for the network edges mechanistically confer activation or inhibition from node $y$ to node $x$.}
    \label{fig:definitions}
\end{figure}

Bias and weighting just indicate deviation away from the unbiased symmetric case as it is difficult to assign a sensible generalised quantitative definition to these qualities. We focus primarily on bias in this manuscript as issues with symmetry can be somewhat mitigated by ensuring that the amplitude $r_{\mathcal{E}}$ is unchanged when comparing activators and inhibitors. It is possible to get a sense of increasing and decreasing bias by plotting the nullcline in the ($x$-$y$) state space that corresponds to $r_{\mathcal{E}}^+(x,y) = r_{\mathcal{E}}^-(x,y)$ (the steady state in $x$ caused by a state variable $y$ in the absence of other edges). We show these nullclines for a symmetric model edge of activator (Fig \ref{fig:dynamics_act}) and inhibitor (Fig \ref{fig:dynamics_inh}) type respectively. The unbiased case can be seen clearly in the solid black curves. Increasing negative bias is shown by the dashed blue curves whereas increasing positive bias is shown by the dashed red curves. Increasing bias shifts the nullcline further to the left or right respectively when compared to the unbiased case in black. Furthermore, the unbiased case necessarily has the feature that the nullclines can be rotated around $(0,0)$ an angle of $\pi$ without changing the nullcline.

 \begin{figure}[H]
      \centering
      %\scalebox{0.9}{
      \begin{subfigure}[b]{0.43\textwidth}
          \centering
          \includegraphics[width=\textwidth]{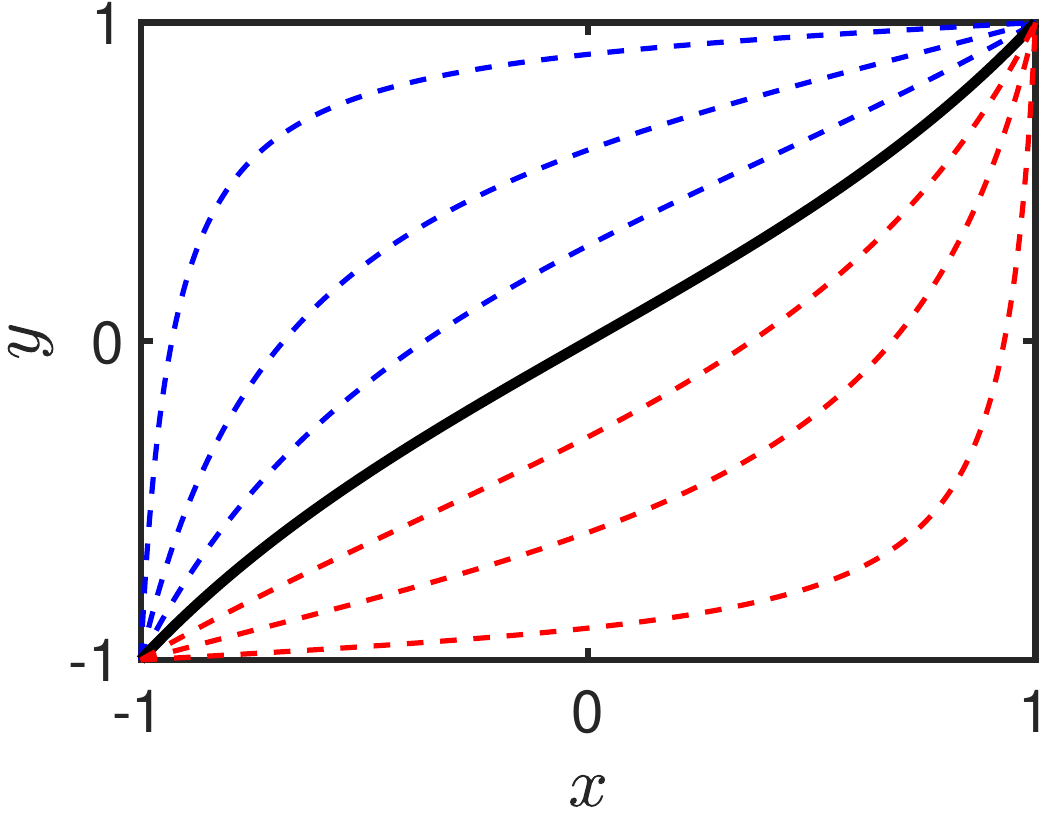}
          \caption{\label{fig:dynamics_act}}
      \end{subfigure}
     \hspace{1cm}
      \begin{subfigure}[b]{0.43\textwidth}
          \centering
          \includegraphics[width=\textwidth]{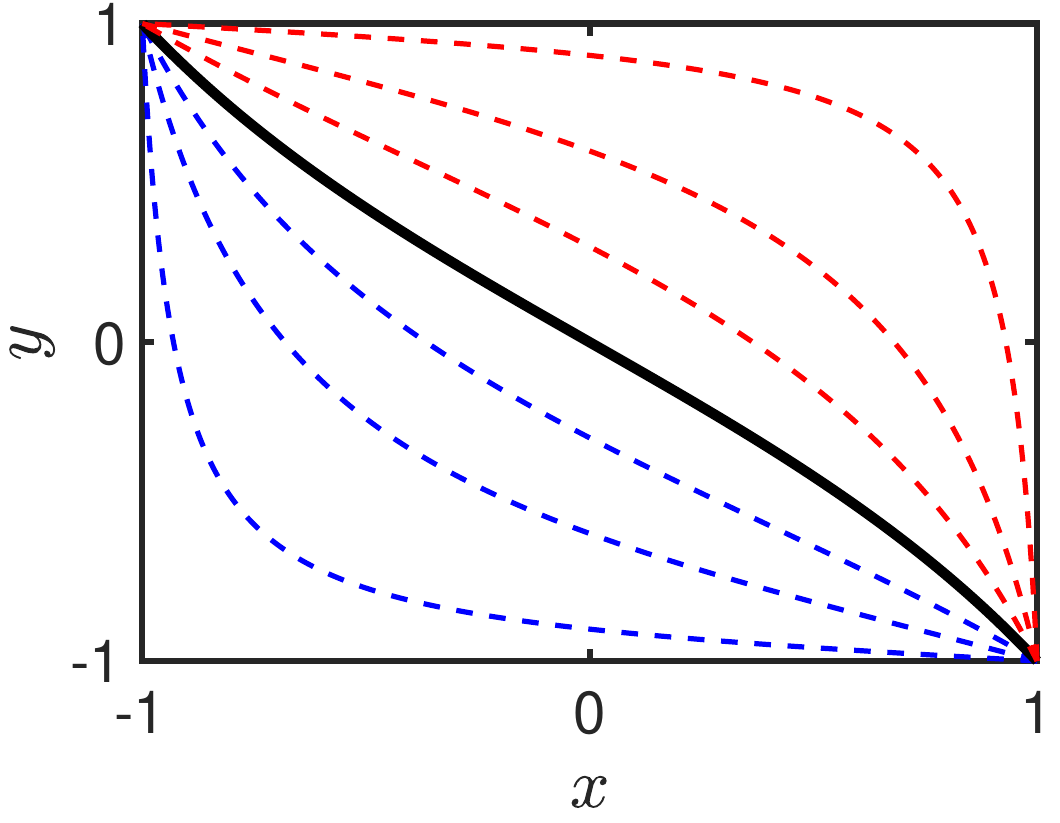}
          \caption{\label{fig:dynamics_inh}}
      \end{subfigure}
     %}
      \caption{Example symmetric activator (a) and inhibitor (b) nullclines defined by  $r_{\mathcal{E}}^{+}(x,y;\mathcal{T})=r_{\mathcal{E}}^{-}(x,y;\mathcal{T})$. The solid black curves corresponds to an example unbiased model. Blue dotted lines show the nullclines for increasingly net negative bias whilst and red dotted lines show the nullclines for increasingly net positive bias.}
 \end{figure}

\subsection{Oriented pathways and pathway similarity}

Consider the example of the simple linear pathway shown in Figure \ref{fig:similarity}. If it is justified that both representations behave the same, we can orient the pathway by a standard `oriented' form. We shall denote the \textit{oriented form} of the pathway as that portrayed on the right of Figure \ref{fig:similarity}; that is, where all relationships are denoted as activators. If we are unable to treat activators and inhibitors as opposites (because the model is biased or assymetric) then it is still qualitatively feasible that the oriented and unoriented pathways of Figure \ref{fig:similarity} (and others) behave `similarly' but not exactly the same. To what degree bias results in dissimilar/divergent oriented and unoriented pathways is the focus of this paper. Here we define how to find the oriented form of a pathway and a network.

\begin{definition}
    A (linear) \textit{pathway} of length $N\geq 1$ is a set of ordered nodes $\{x_i\}_{i=1}^{N}$ connected by $N-1$ edges $\{\mathcal{E}_i\}_{i=1}^{N-1}$ where $\mathcal{E}_i$ connects $x_i$ to $x_{i+1}$. The pathway must have a prescribed \textit{input} node $x_1$ and \textit{output} node $x_N$. All non-output nodes are \textit{internal} nodes.
\end{definition}

\begin{definition}
    A \textit{network} is a set of one or more pathways possibly connected by cross-talk and/or connected to external edges (edges not associated with a pathway/external stimuli).
\end{definition}

\begin{definition}
    Two networks are \textit{equivalent} (or \textit{similar}) if all constituent pathway output nodes have identical (or similar) response outputs.
\end{definition}

\begin{definition}
    An \textit{oriented} pathway is one which is equivalent or similar to any given pathway but contains all activation edges. An \textit{oriented} network is one that contains all oriented pathways. A pathway or network is called \textit{orientable} if it has an oriented form, which is equivalent.
\end{definition}

These definitions bring us to a key result expressed in the following lemma and generalised in the proceeding theorem.

\begin{lemma}\label{thm1}
All unbiased and symmetric pathways are orientable.
\end{lemma}

\begin{proof}
To prove Lemma \ref{thm1}, we first demonstrate an interesting property of unbiased and symmetric pathway/network models (\ref{rreswing}). 
Consider any node $x_i$ which may have any number of edges pointed towards it $\mathcal{E}_i^{(\text{in})}$ and out of it $\mathcal{E}_i^{(\text{out})}$. We have
\begin{align}
    \dot{x}_i  &= \sum_{\mathcal{E}_{ki}\in \mathcal{E}_{i}^{(\text{in})}} r^{+}_{\mathcal{E}_{ki}}(x_i,x_k;\mathcal{T}_{ki}) - r^{-}_{\mathcal{E}_{ki}}(x_i,x_k;\mathcal{T}_{ki}), \quad \text{and} \\
    \dot{x}_{j} &= r^+_{\mathcal{E}_{ij}}(x_{j},x_{i};\mathcal{T}_{ij}) - r^-_{\mathcal{E}_{ij}}(x_{j},x_{i};\mathcal{T}_{ij})+ \psi_{ij}', \quad \text{for each } \mathcal{E}_{ij}\in \mathcal{E}_i^{(\text{out})},
\end{align} 

where $\mathcal{E}_{ki}$ is an edge of type $\mathcal{T}_{ki}$ that points from a node $x_k$ to $x_i$ and $\psi_{ij}'$ are due to edges that affect a node $x_j$ but are not connected to the node $x_i$. Here we have all terms which include the node $x_i$ in the model. We shall now perform a \textit{flip} of the node $x_i$ by changing its state variable to its complements. That is, everywhere, we shall write these equations replacing $x_i$ for $-x_i$. 
 \begin{align}
    \dot{x}^{\text{(flip i)}}_i  &= \sum_{\mathcal{E}_{ki}\in \mathcal{E}_i^{(\text{in})}} - r^+_{\mathcal{E}_{ki}}(-x_i,x_k;\mathcal{T}_{ki}) + r^-_{\mathcal{E}_{ki}}(-x_i,x_k;\mathcal{T}_{ki}), \quad \text{and} \\
    \dot{x}_j^{\text{(flip i)}} &= r^+_{\mathcal{E}_{ij}}(x_j,-x_i;\mathcal{T}_{ij}) - r^-_{\mathcal{E}_{ij}}(x_j,-x_i;\mathcal{T}_{ij})+ \psi_{ij}', \quad \text{for each } \mathcal{E}_{ij}\in \mathcal{E}_i^{(\text{out})}
\end{align}
Consider now coupling this flipping of node $x_i$ operation with the operation of also performing a \textit{flip} in the type of all edges in $\mathcal{E}_i^{(\text{in})}$ or $\mathcal{E}_i^{(\text{out})}$. That is, we change respective types for all edges connected to node $i$ $\mathcal{T}\rightarrow \mathcal{T}^*$ (which represents interchanging $\mathcal{A}\rightarrow \mathcal{I}$ and $\mathcal{I}\rightarrow \mathcal{A}$). We notice that if the model is symmetric, $\dot{x}_j^{\text{(flip i)}}$ reverts back to $\dot{x}_j$ after implementation of Definition \ref{symedge}. On the other hand, since the model is also unbiased, after implementation of Definition \ref{biasedge}, $\dot{x}^{\text{(flip i)}}_i$ becomes $-\dot{x}_i$ as expected after flipping both the node $x_i$ and all connected edges. Therefore, if a model is unbiased and symmetric then flipping any node $x_i$ along with all of its connecting (in and out) edges from activation to inhibition and vice versa will leave all other nodes in the model completely unchanged. We continue this proof as a special case of Theorem \ref{thm2}.

\end{proof}

\begin{theorem}\label{thm2}
All unbiased and symmetric networks are orientable.
\end{theorem}

\begin{proof}

Lemma \ref{thm1}, and by trivial extension Theorem \ref{thm2}, can be proven by induction. We prove this by showing orientability of a single pathway noting that a network of pathways can be approached in the same way. The induction base case for a pathway is where $N=2$ (a pathway with a single edge). The pathway consists of the output node $x_i$ and input node, $x_j$ as shown in Fig. \ref{fig:pf1}. If the edge connecting $x_j$ to $x_i$ is an activation, then it is already in oriented form. If not, then, as we have shown, it is possible to \textit{flip} the node $x_j$ along with all edges connected to $x_j$ without changing the output in the manner just described. The subsequent oriented pathway is therefore equivalent to the original, and thus a pathway of length $N=2$ is orientable if the model is symmetric and unbiased. It is important here that in performing the flipping of the node $x_j$ and the edge that the node $x_i$ is left unchanged (despite the fact that $x_j$ has become $-x_j$) because $x_j$ is an internal node whilst $x_i$ is the output node that needs to be identical if equivalence is to be achieved between the oriented and unoriented pathways.
 \begin{figure}[H]
 \centering
 \includegraphics[width=0.3\textwidth]{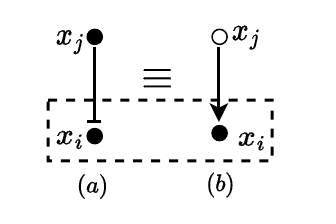}
 \caption{\label{fig:pf1} A figure showing the equivalence in (a) the signalling pathway with $N=2$ and a single inhibition edge with (b) its oriented form. A hollow node is used to indicate that it is the complement (negative/flipped) of the original state variable for this node. The output nodes are in a dashed box to indicate that these nodes are unchanged between the unoriented and oriented pathways making these pathways equivalent.}
 \end{figure}

We now assume that any unbiased symmetric pathway of length $k\geq 2$ is orientable. Consider a pathway of length $N=k+1$ with ordered nodes $\{x_i \}_{i=1}^{k+1}$. This pathway contains within it a pathway of length $k$ with input node $x_2$ and output node $x_{k+1}$. Since the pathway of length $k$ is orientable, we may perform any necessary flips of internal nodes and edges to generate the oriented form of the pathway from node $x_2$ to $x_{k+1}$. At this point, we are now permitted to flip node $x_1$ and all connected edges to place the full pathway into oriented form if the edge from $x_1$ to $x_2$ is an inhibition edge (otherwise the full pathway is already in oriented form). Thus, if a pathway of length $k$ is orientable then so is a pathway of length $k+1$. Application of this process to all pathways in a network proves Theorem \ref{thm2}.
\end{proof}

Inspired by the proof of Theorem \ref{thm2}, the follow algorithm is used to orient any given pathway.

\begin{algorithm}[H]
\SetAlgoLined
\KwData{A signalling pathway with $N$ nodes $\{x_i\}_{i=1}^{N}$ and $N-1$ edges $\{\mathcal{E}_i\}_{i=1}^{N-1}$}
\KwResult{Finding the associated oriented pathway}
$x_i \leftarrow x_N$ \Comment{Start at the $N^{th}$ node} \; \\
\For{all edges $\mathcal{E}_i$ from $\mathcal{E}_N$ to $\mathcal{E}_2$}{
\If{$\mathcal{E}_{i-1}$ is an inhibition}{
    {change node $x_{i-1}$ by flipping all connected edges\;
    }
}
    $x_i\leftarrow x_{i-1}$
}
\caption{Finding an oriented pathway}
\label{algorithm}
\end{algorithm}

The steps outlined in Algorithm \ref{algorithm} are shown in Figure \ref{fig:algo1}.

\begin{figure}[H]
\centering
\includegraphics[width=0.9\textwidth]{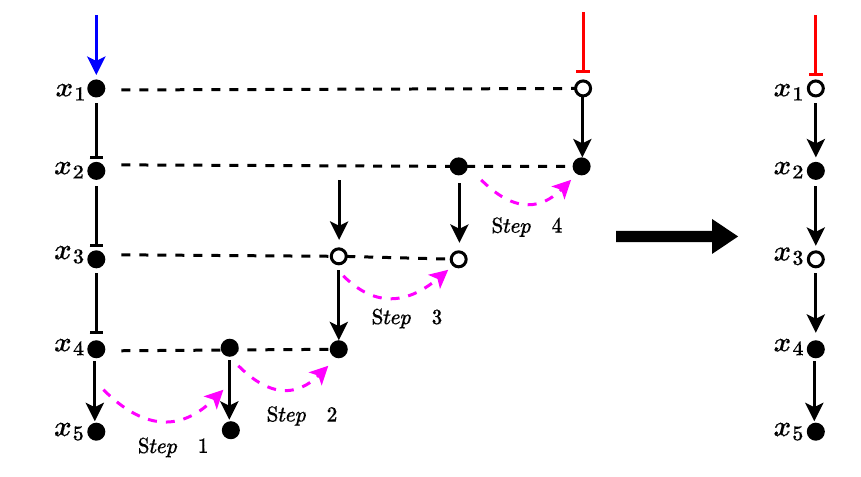}
\caption{An example of how to convert a pathway into its oriented form using the steps outlined in Algorithm \ref{algorithm}}
\label{fig:algo1}
\end{figure}

Generalisation to orienting a full network is done using Algorithm \ref{algorithm2}.

\begin{algorithm}[H]
\SetAlgoLined
\KwData{A signalling network.}
\KwResult{Finding the associated oriented network}
Identify all constituent signalling pathways comprising the network. \\
\For{all pathways}{
Perform Algorithm \ref{algorithm}.
} 
\caption{Finding a oriented network.}
\label{algorithm2}
\end{algorithm}

% %--------------------------------------------------------------------------------------------
% %---------------------------------- END OF Methods---------------------------------------------
% %--------------------------------------------------------------------------------------------

%--------------------------------------------------------------------------------------------
%---------------------------------- RESULTS  ---------------------------------------------
%--------------------------------------------------------------------------------------------
\section{Numerical Results and Discussion}
\label{sec:Numerics}

Evident by the prevalence of the idea of protein `activation' and `deactivation', most biochemical network models have bias. We conduct a numerical exploration of positive and negative bias and the robustness of the oriented form approximation. We focus first on single pathways and then explore the effect of feedback and network cross-talks. As expected, making the networks more complicated makes the pathways more idiosyncratic (diverges away from the behaviour of a common oriented form). However, for a given network, the oriented form can be much more robust if the bias is positive or negative depending on the situation. Furthermore, we find that for simple networks robustness is strong as bias is increased but beyond a critical bias robustness is lost rapidly and the oriented form does not represent the unoriented network; in some cases even behaving in an opposite way.

\subsection{Test model}\label{testmodel}

We use a caricature test model so that we can increase or decrease bias explicitly. The model is parameterised by three parameters: $\alpha$, $\beta$, and $\phi$. The parameter $\alpha$ describes the amplitude (timescale), $\beta$ prescribes variable nonlinearity, and $\phi$ is a proxy for the model's bias and is chosen independently for each edge. This model is unlikely to represent a real biological system. That being said, by the manipulation of parameters $\alpha$, $\beta$ and $\phi$ there is a lot of flexibility in the model functions $r_\mathcal{E}^\pm(x,y;\mathcal{T}_\mathcal{E})$. In this sense, the test model allows for generic observations in the role of bias in the approximation that inhibition is the opposite of activation in the context of increasingly complex networks. The rates $r_{\mathcal{E}}^+$ and $r_{\mathcal{E}}^-$ are defined as follows: 
% using (\ref{eq:twoway_eq1_act})-(\ref{eq:twoway_eq1_in2}). 

For activation, 
\begin{flalign}
    &r_{\mathcal{E}}^{+}(x,y;\mathcal{A}) = \Biggl( \frac{1+\phi}{4}\Biggl)(1+y) F(x;\alpha,\beta)\label{eq:twoway_act_mod1} \\
    &r_{\mathcal{E}}^{-}(x,y;\mathcal{A}) = \Biggl( \frac{1-\phi}{4}\Biggl)(1-y)F(-x;\alpha,\beta)\label{eq:twoway_act_mod2}.
\end{flalign}
For inhibition,
\begin{flalign}
    &r_{\mathcal{E}}^{+}(x,y;\mathcal{I}) = \Biggl( \frac{1+\phi}{4}\Biggl)(1-y) F(x;\alpha,\beta)\label{eq:twoway_act_mod3} \\
    &r_{\mathcal{E}}^{-}(x,y;\mathcal{I}) = \Biggl( \frac{1-\phi}{4}\Biggl)(1+y)F(-x;\alpha,\beta)\label{eq:twoway_act_mod4}.
\end{flalign}

where, 
\begin{equation}
F(x;\alpha,\beta) = \frac{\alpha\beta(1-x)}{2\beta-(1+x)}.
\end{equation}

Here $\phi = 0$ indicates no bias and $\phi>0$ ($\phi<0$) indicates positive (negative) bias. The function $F$ is based on Michaelis-Menten enzyme kinetics but translated to our state variable framework (where state variables vary between -1 and 1). Our study uses MATLAB's ODE45 solver to run simulations. Specifically, where not stated, for each bias value $\phi$, we test a total of 150 diverse sets of $\alpha$ and $\beta$ values to better isolate behaviour attributable to bias and network topology rather than from non-linearity and timescale.

\subsection{Linear pathways}
\subsubsection{Unbiased linear pathways}

We begin by investigating single linear pathways of up to 5 nodes activated at the input node. Node 1 is the input node and Node 5 is the output node.

In Fig. \ref{fig:comparepathway} we compare the time evolution of each node for a sample pathway against its oriented form obtained by Algorithm \ref{algorithm}. In this particular case, we demonstrate numerically Theorem \ref{thm2} by setting the bias $\phi = 0$. The observation here is that the outputs (orange) are identical between pathway and oriented pathway whilst all flipping of internal nodes that is done in the orienting process correctly negates the time evolution of the node at all times.

It is clear that when there is no (or very little) bias in the symmetric model, observing two inhibitions in series is quantitatively identical (or very similar) to two activators. This is because, at least  qualitatively, it is understood that inhibiting an inhibitor is akin to a double negative; resulting in a positive. Understanding this equivalence (and when it is appropriate to assume it) makes it easier to qualitatively assess the pathway (and network) behaviour from generic topological relationships -- as is common with biochemical network models. 

\begin{figure}[H]
\centering
\includegraphics[width=0.71\textwidth]{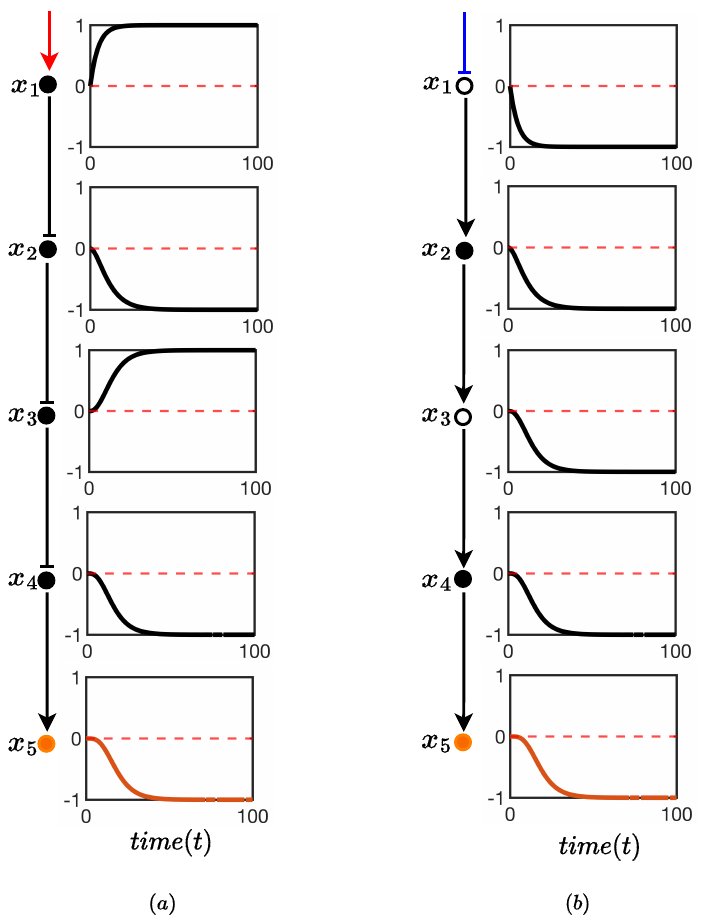}
\caption{(a) A cellular biochemical signalling pathway consisting of five nodes $x_i$ for $i=1,2,3,4,5$. The black circles and connectors show the pathway, and the orange node $x_5$ indicates the output node of the pathway. An external stimulus is indicated by a red arrow, which activates the node $x_1$. The (symmetric) test model (Section \ref{testmodel}) is used to model the pathway edges. In this model, each edge of the network is associated with two parameters, $\alpha_{ij}$ and $\beta_{ij}$ where $i$ is the source and $j$ is the destination of the connector between two nodes.  The parameter values for $\alpha$ and $\beta$ were chosen arbitrarily as $\alpha = \{\alpha_{12},\alpha_{23},\alpha_{34},\alpha_{45}\}= \{0.5,1,2,1.3\}$ and $\beta = \{\beta_{12},\beta_{23},\beta_{34},\beta_{45}\}= \{10,5,15,18\}$. A boxes on the right side each node displays the state variables of each node from $t=0$ to $t=100$. All states are initialised in the neutral position of $0$. 
The pathway in (b) is the oriented representation of the signalling pathway shown in (a). The hollow nodes represent nodes that are flipped in the orienting process described by Algorithm \ref{algorithm}. Since this process results in a flipped $x_1$ the input edge should be interpreted as an inhibition which is indicated in blue. In choosing $\phi = 0$ (no bias) and noting that $x_5$ is the same for (a) and (b) we have verified the result of Theorem \ref{thm2}.}
\label{fig:comparepathway}
\end{figure}

\subsubsection{Biased linear pathways}

We now shift our investigation to the robustness of the oriented pathway reduction of single pathways with bias. To compare a pathway/network against its oriented form we measure the similarity between the two by focusing on two metrics. In both metrics we initialise the pathway and its oriented form in the neutral position (all node states equal to zero). We then run the model for both oriented and unoriented forms numerically until steady state. The first metric comparing the unoriented to oriented descriptions is the difference in the long term behaviour. We have restricted our study to non-oscillatory systems which limit to fixed steady states over time. We define the steady state difference to be
\begin{equation}\label{deltass}
    \delta_{ss} = \left< \delta - \bar{\delta}\right>_{\alpha,\beta} =  \left<\lim_{t\rightarrow \infty} x_N(t) - \lim_{t\rightarrow \infty} \bar{x}_N(t)\right>_{\alpha,\beta} ,
\end{equation}  
where $\delta$ and $\bar{\delta}$ are the unoriented and oriented steady states of the output node $x_N$, respectively. The bar notation here corresponds to the oriented form. The brackets $\left<\cdot \right>_{\alpha,\beta} $ indicate averaging over many combinations of parameter sets $\alpha$ and $\beta$ describing the edges in the model.

A second metric is defined to measure the difference in the transient aspects of the response of the unoriented and oriented cases. It is constructed by first normalising each output node by its steady state and then integrating the difference between the unoriented and oriented transient output values. The resultant error is positive (negative) if the unoriented form approaches steady state faster (slower) than the oriented form. Very large values in this metric may indicate the presence of some temporal behaviour not present in the other form (for example, rebounding behaviour). This is especially the case if one of the forms responds to the input and returns close to the neutral state over time. 

\begin{equation}
\delta_\tau = \left<\int_0^\infty \left( \frac{x_N(t)}{\delta } - \frac{\bar{x}_N(t)}{\bar{\delta}} \right) \ \mathrm{d}t \right>_{\alpha,\beta}
\end{equation}

Behavioural differences between unoriented and oriented pathways can also lead to small errors $\delta_\tau$. The purpose of $\delta_\tau$ is to get insight into response times of unoriented and oriented pathways under simple conditions. This metric is less useful in the case where responses are sufficiently different in nature.

Fig. \ref{fig:delta_ss} displays the steady-state error $\delta_{ss}$ for all possible linear pathways of length $N=5$. Each chart in the figure is a plot of $\delta_{ss}$ versus $\phi$ in the model. That is, left of centre represented increasing negative bias and right of centre represents increasing positive bias. The pathways in the figure are shown at the top of each chart and immediately below these pathways are the oriented forms (showing specifically which nodes are flipped in the orienting process). The charts themselves are strategically organised into columns based on the total number of nodes that are flipped in order to create the oriented form (one node on the left, two nodes in the centre, three and four nodes on the right). From top to bottom, we order the charts such that in the orienting process flips that are generally more upstream are at the top whereas downstream flips are generally down the bottom. We also plot each chart in red if, after the orienting process, the input stimulus does not change type from activation but in blue if the input stimulus changes from activation to inhibition.

Setting the charts out like that in Fig. \ref{fig:delta_ss} shows a number of interesting and nontrivial properties for linear pathways:
\begin{enumerate}
    \item Locally around the unbiased case $\phi = 0$ we see good agreement in steady states between the unoriented and oriented pathways.
    \item Agreement with the oriented pathway is very robust if the model is \textit{negatively} biased \text{and} the oriented form requires a flip in the input stimulus from activation to inhibition. If there is no flip in the input stimulus then agreement is very robust if the model is \textit{positively} biased.
    \item As bias is increased, agreement with the oriented pathway remains robust up until some critical level of bias at which stage agreement rapidly decreases until under some conditions the worst cases reach $\delta_{ss} \approx \pm 2$ indicating that the oriented and unoriented forms of the pathway limit towards extreme opposite outputs.
    \item The critical bias, before disagreement with the oriented form rapidly increases, is closer to the unbiased case -- that is, the oriented form is less robust -- if the orienting process flips over nodes in higher quantities \textit{or} further downstream. Pathways which satisfy these criterion also have more extreme disagreements with their oriented form.
\end{enumerate}

These observations were also consistent with results found for pathways of lengths $N=3$ and $N=4$ (not shown here).

%Steady state error of pathways with 5 nodes

\begin{figure}[H]
    \centering
    \includegraphics[width=.8\linewidth]{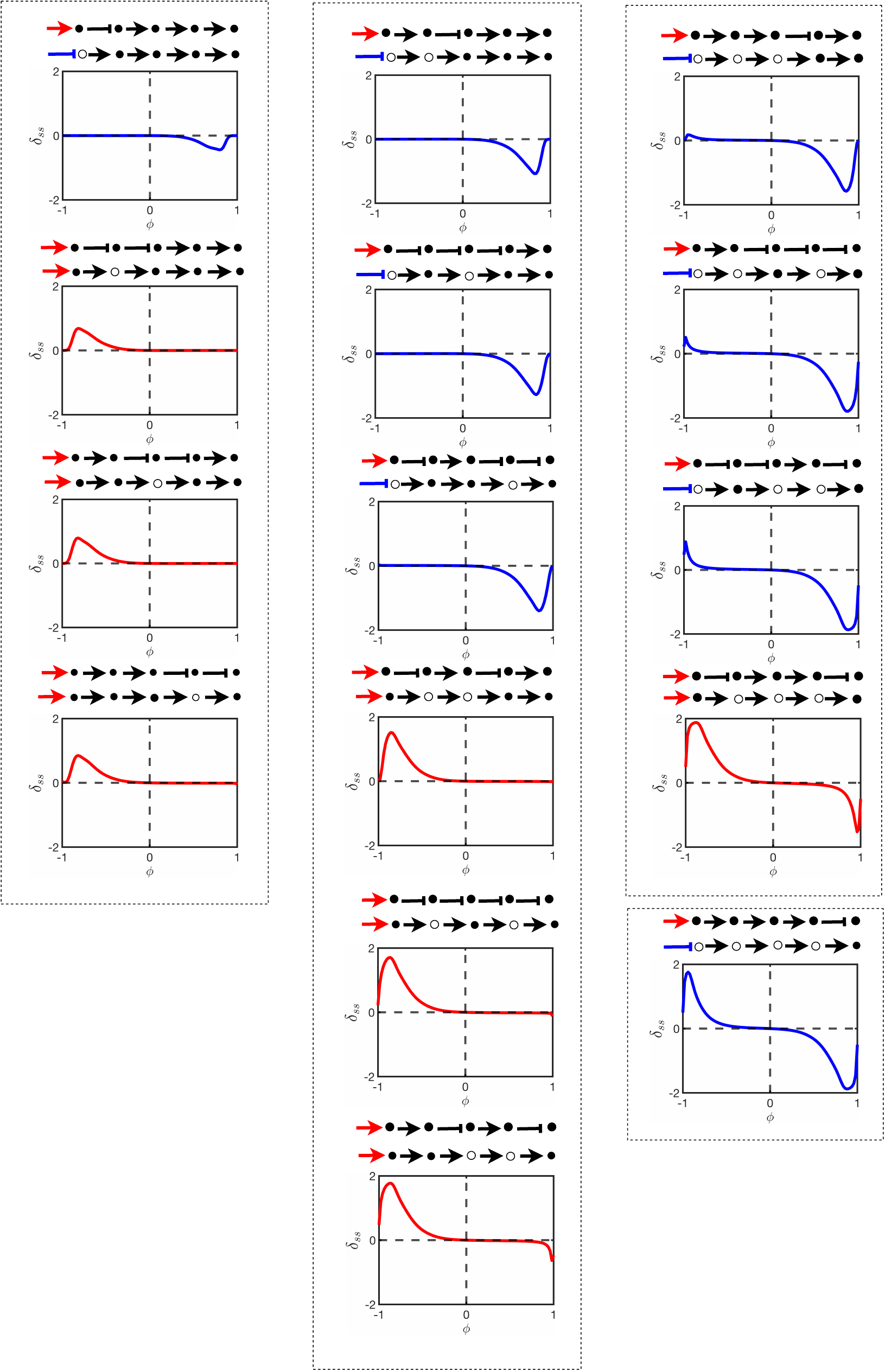} 
     \caption{This figure displays the steady-state error $\delta_{ss}$ of different signalling pathways of length $N=5$ as defined by Equation (\ref{deltass}). The model used is the test model in Section \ref{testmodel}. The error is plotted in charts as functions of the bias parameter $\phi$ and the pathway with its oriented form are shown above each chart. The colors of the plots correspond to flipping (blue) or no flipping (red) of the input stimulus as a result of the orienting process. The columns separate the number of flipped nodes in the orienting process whilst each column is ordered from top to bottom where these flips occur more upstream or downstream respectively.}
    \label{fig:delta_ss}
\end{figure}

It is also interesting to observe that disagreement between the oriented and unoriented forms as a result of bias compounds as the number of flips increases. This can be seen in Fig. \ref{fig:all_in_one} where the sum of all errors for each of the four single flip pathways (left column) in Fig. \ref{fig:delta_ss} are plotted against the steady state error associated with the pathway which requires all four flips to orient (bottom right chart in Fig. \ref{fig:delta_ss}). The later error eclipses the sum of the former errors however in both cases robustness remains fairly strong for a significant interval around $\phi = 0$ (the unbiased case).

\begin{figure}[H]
    \centering
    \includegraphics[width=.6\linewidth]{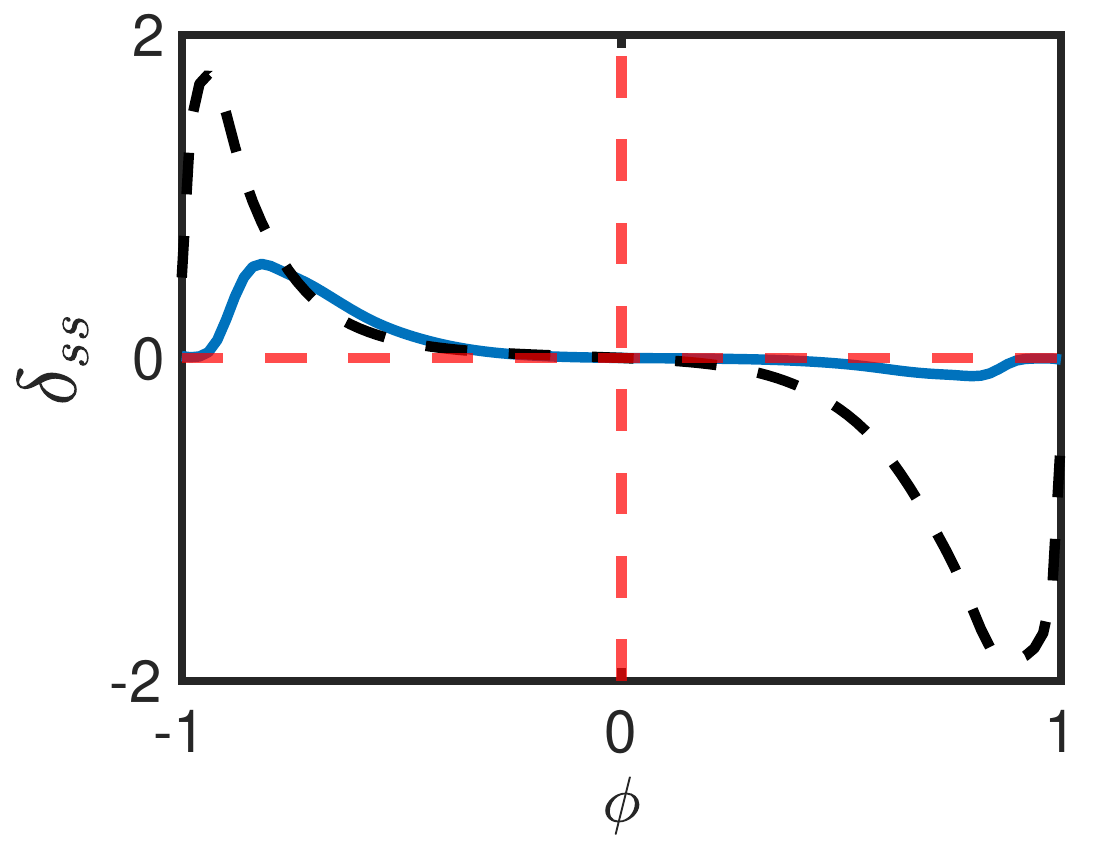}  
    \caption{The figure shows the sum of steady-state errors $\delta_{ss}$ of all pathways of length $N=5$ with one-node flips in their oriented form (blue) against the steady state error of a single pathway with and four nodes flipped in its oriented form (black dashed).}
    \label{fig:all_in_one}
\end{figure}

Using the same charting order used in Fig. \ref{fig:delta_ss}, we chart the errors $\delta_\tau$ in Fig. \ref{fig:pathway_dtau}. The plots of $\delta_\tau$ show similar behaviour under changes in bias to $\delta_{ss}$. This is especially the case for pathways requiring few flips in orienting. In this case, oriented and unoriented pathways which tend towards different steady states also take different rates in getting there. Pathways requiring many flips produce more complex behaviour over time and $\delta_\tau$ is less insightful.

\begin{figure}[H]
    \centering
    \includegraphics[width=.9\linewidth]{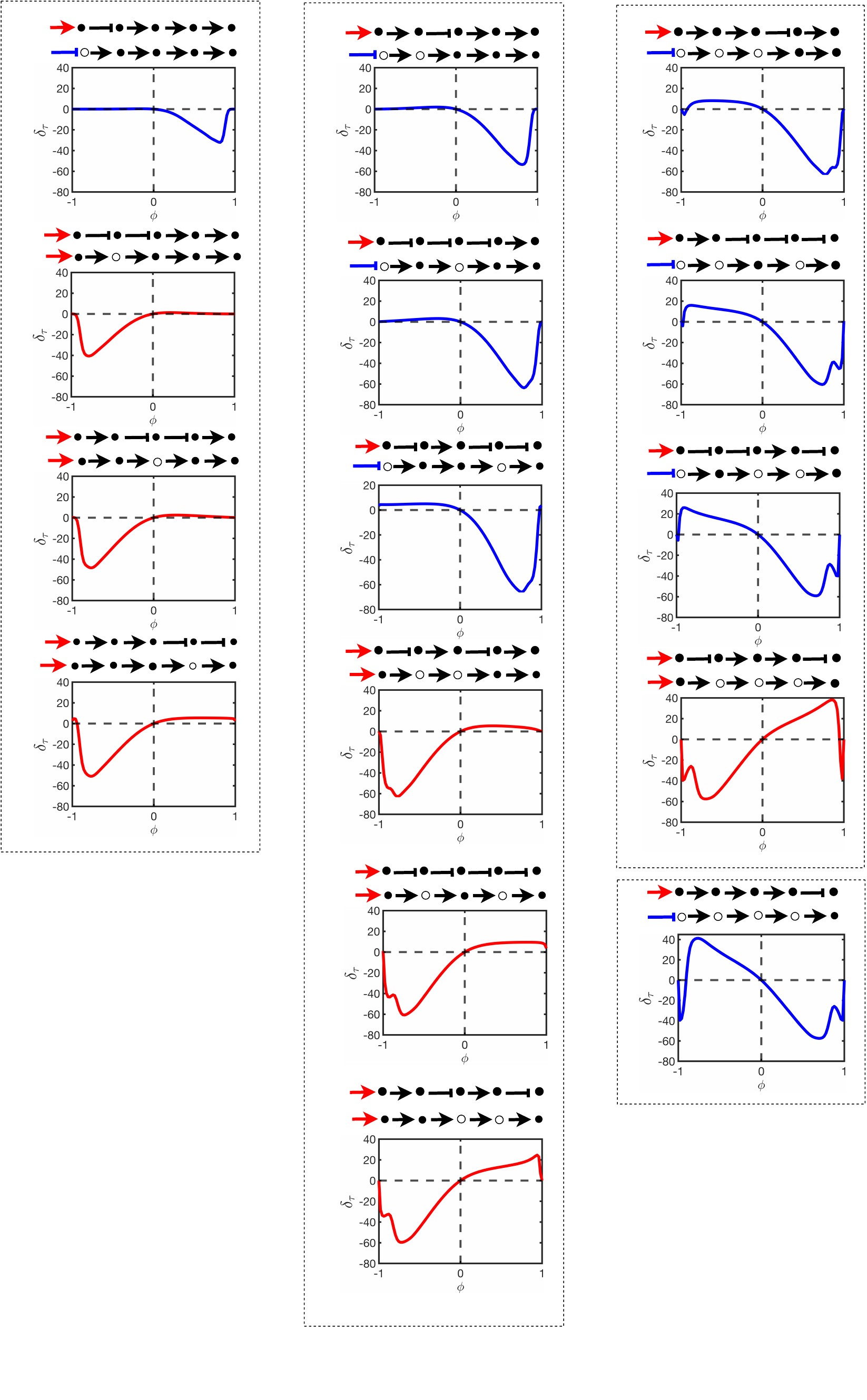}  
    \caption{This figure displays the steady-state error $\delta_{ss}$ of different signalling pathways of length $N=5$ as defined by Equation (\ref{deltass}). The model used is the test model in Section \ref{testmodel}. The error is plotted in charts as functions of the bias parameter $\phi$ and the pathway with its oriented form are shown above each chart. The colors of the plots correspond to flipping (blue) or no flipping (red) of the input stimulus as a result of the orienting process. The chart arrangement is the same as that of Fig. \ref{fig:delta_ss}.}
    \label{fig:pathway_dtau}
\end{figure}

In the following section we shift our attention to the affect of bias on the oriented form on networks. We begin with simple pathways with feedback before looking at an example network with significant integration of multiple pathways.

%----------------------------------------------------------------------------------------
\subsection{Networks}

\subsubsection{Unbiased pathways with feedback}

Investigation of linear pathways in the previous section shows that care should be taken when assuming that inhibitors act as diametrically opposite activators under biased conditions in the case where there this assumption is taken in multiple instances (where there are many flips to get to the oriented form). We now attempt to investigate the validity of this assumption under increasing complexity and in particular in the case of feedback.

In this section, we will fix a linear pathway of length $N=5$ consisting of 3 inhibitions and ending in a single activation. The orienting procedure then requires a flip of node 1 and 3; in the process flipping the input stimulus from activator to inhibitor. We then add a single feedback in the form of an inhibition to this pathway.
 
Fig. \ref{fig:comparefeedback}, like Fig. \ref{fig:comparepathway}, compares the evolution of each node in this pathway under unbiased conditions in using the test model in Section \ref{testmodel} with an inhibition feedback from the output node 5 to node 3. The purpose is to validate Theorem \ref{thm2}; noting that output nodes are exactly the same between oriented and unoriented pathways.

\begin{figure}[H]
\centering
\includegraphics[width=0.8\textwidth]{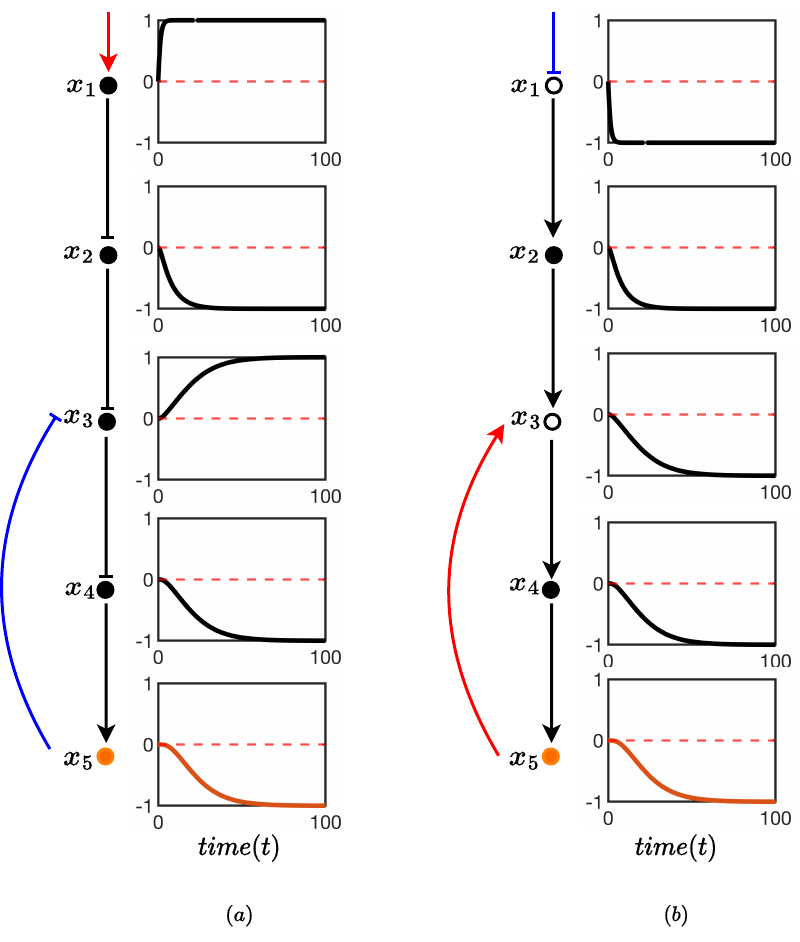}
 \caption{Here we compare an unoriented pathway with feedback (a) against its oriented form (b) for the test model in Section \ref{testmodel} under unbiased conditions ($\phi = 0$). The parameters and formatting used is the same as that in Fig. \ref{fig:comparepathway}. The figure shows equivalency in the two networks and validates Theorem \ref{thm2}.}
\label{fig:comparefeedback}
\end{figure}

Figs \ref{fig:dss_feedback} and \ref{fig:dtau_feedback} show an array of pathways of length $N=5$. These pathways all require two flips at nodes 1 and 3 to orient. The figures show how the errors $\delta_{ss}$ (Fig. \ref{fig:dss_feedback} showing error in steady state) and $\delta_\tau$ (Fig. \ref{fig:dtau_feedback} describing discrepancy in temporal behaviour) are influenced by bias. In each figure, the array consists of the same pathway subject to all combinations of inhibition feedback. The first column consists of feedback a distance of a single node, the second column, a feedback a distance of two nodes, and in the last column a feedback is a distance of three and four nodes. The top charts display feedback which is introduced further downstream than the charts at the bottom. The pathways are oriented the same in all cases but the process leaves the feedback sometimes as an activator (red) and sometimes as an inhibitor (blue) depending on its location in the pathway. The charts are plotted in these colours respectively to highlight this feature across the chart which ultimately determines if the feedback is a negative or positive feedback.

\begin{figure}[H]
\centering
\includegraphics[width=0.8\textwidth]{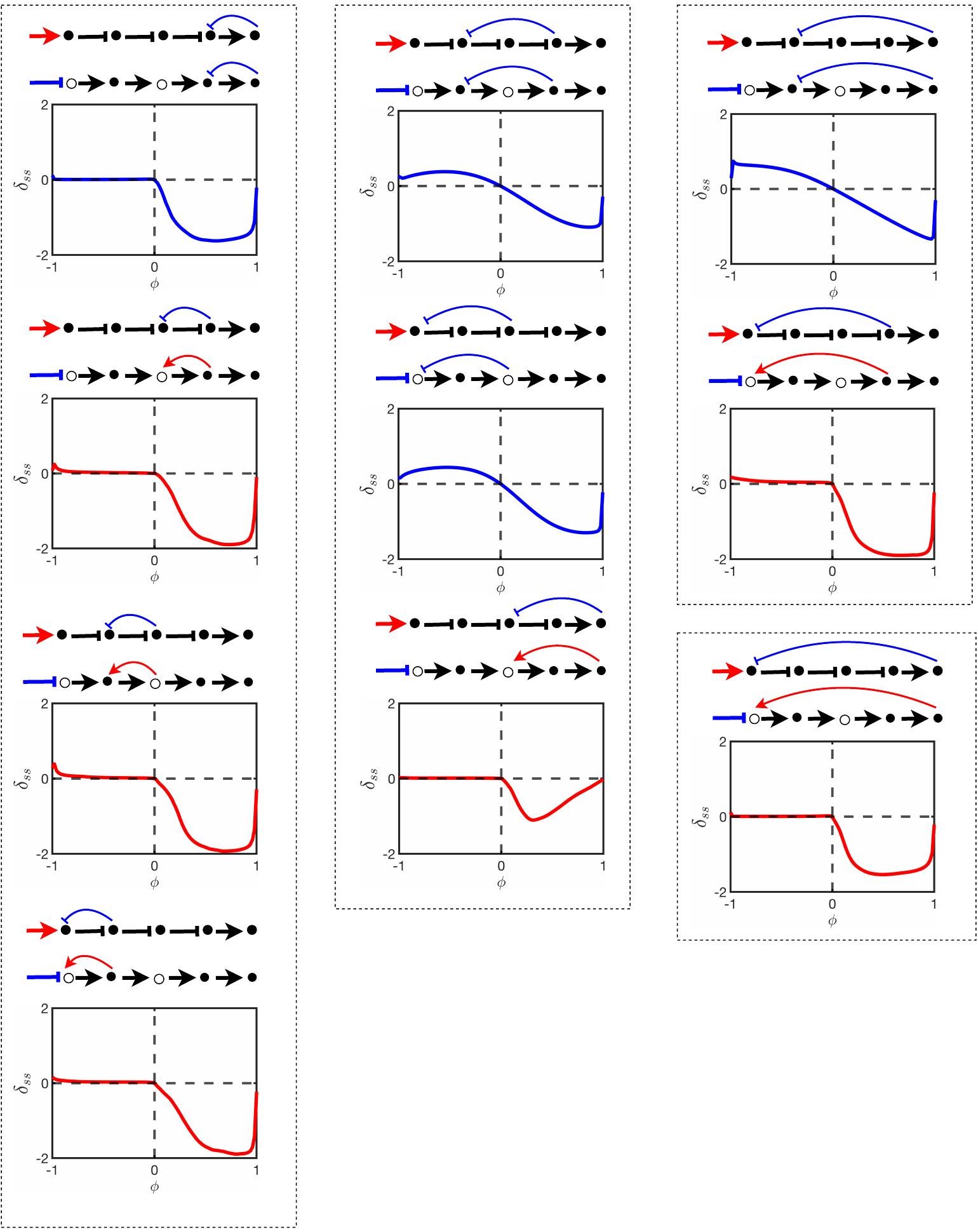}
 \caption{The figure shows the steady-state error of different signalling pathways with five nodes, as determined by varying $\phi$  values under an external stimulus. The y-axis on each chart represents the steady-state error ranging from -2 to 2, while the x-axis shows $\phi$ values ranging from -1 to 1. The extremes in each graph indicate where the non-oriented and oriented pathways disagree significantly, and the flat line portion of the chart indicates the range of robustness in the steady state under increasing levels of bias.}
\label{fig:dss_feedback}
\end{figure}

\begin{figure}[H]
\centering
\includegraphics[width=0.8\textwidth]{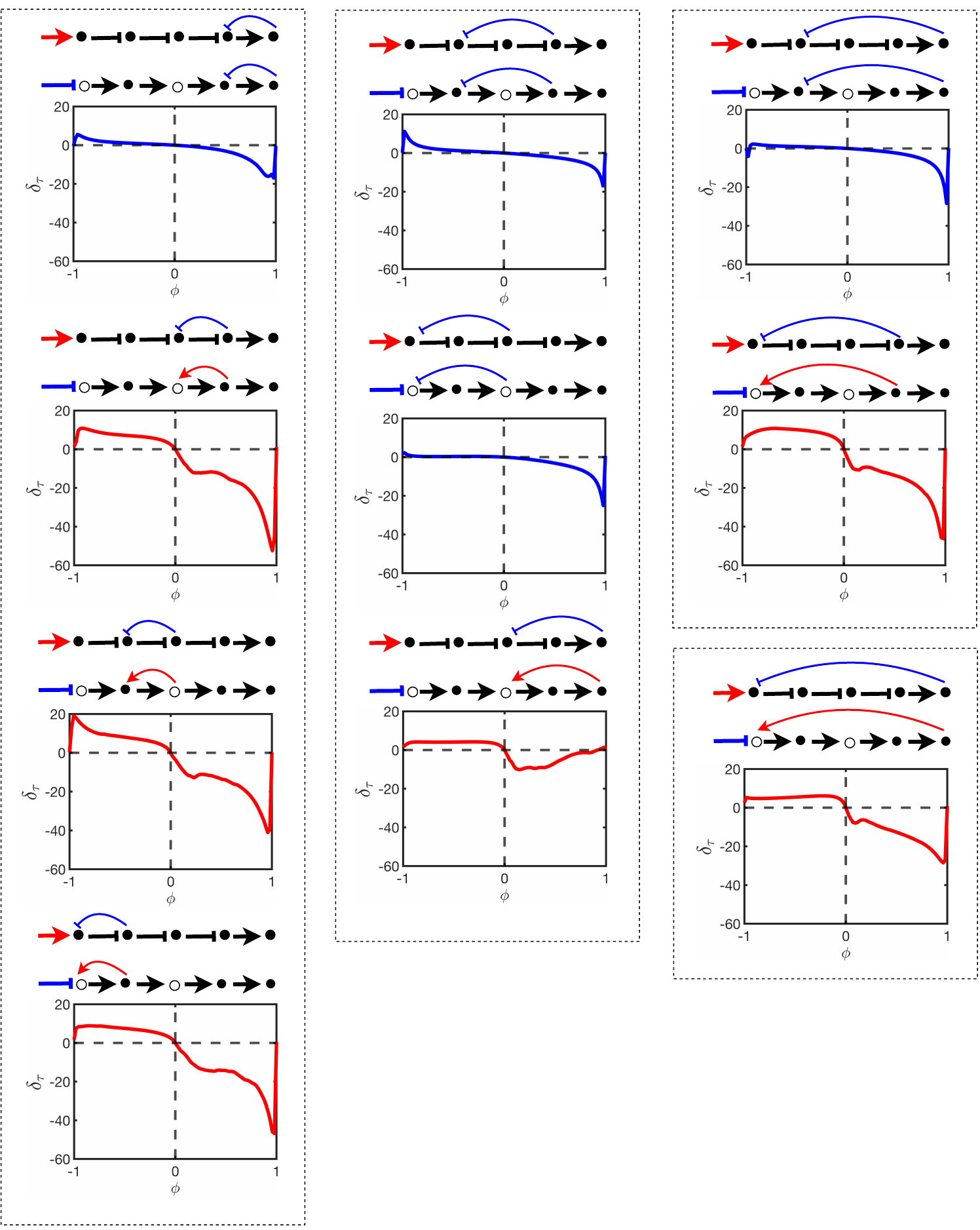}
 \caption{This figure shows the $\delta_{\tau}$ for the same pathways in Fig. \ref{fig:dss_feedback}, as determined by varying $\phi$  values under an external stimulus. The y-axis represents $\delta_{\tau}$, while the x-axis shows $\phi$ values ranging from -1 to 1. The figure arrangement is the same as that of Fig. \ref{fig:dss_feedback}.}
\label{fig:dtau_feedback}
\end{figure}

The following key observations are found across our numerical investigation and exemplified in Figs \ref{fig:dss_feedback} and \ref{fig:dtau_feedback}.
\begin{enumerate}
    \item Orientability is more robust with the addition of negative (or positive) bias than the alternative if the oriented pathway is inhibited (or activated, respectively). This was found previously in the case of a simple linear pathway but this phenomena seems to be also valid with the introduction of complexity such as feedback (and later we will show this to be also the case for a significantly more complex network). Feedback tends to even amplify the significance of negative versus positive bias, especially in the case of the steady state (and excepting for our next observation).
    \item Orientability (in the \textit{steady state} but not the temporal behaviour) even close to the unbiased case is compromised in the specific case where (a) there is negative feedback (blue curves in Fig \ref{fig:dss_feedback}) \textit{and} (b) when the feedback loop includes at least one node which is flipped in the orienting process.
    \item Orientability (in the \textit{temporal behaviour} but not the steady state) even close to the unbiased case is compromised under positive feedback (red curves in  Fig. \ref{fig:dtau_feedback}) but not under negative feedback (red curves). This compromise is still more pronounced in the case of positive bias due to the inhibition stimulation of the pathway. 
\end{enumerate}

\subsection{Network complexity and pathway cross-talk}
%Does this depend on the position of the crosstalk in the network.

As it is difficult to do a systematic numerical study of complex network topologies and pathway crosstalk we will instead only look anecdotally at a realistic network and in particular the EGFR/HER2 signalling network model published by Yamaguchi \textit{et al.} in 2014 \cite{yamaguchi2014signaling}. Using Algorithm \ref{algorithm2} to orient a general network requires first to assign nodes to pathways (as defined explicitly in this manuscript not necessarily in a general biological sense) based on key flows of information. Each node must belong to a pathway and each pathway has to have an input node and an output node with edges directed from input to output. Of course, this leaves the choice of pathways an open problem with many solutions. We attempt to do this in such a way that makes the most amount of sense biologically and favoring input nodes as nodes with external stimulus and aligning as much as possible to conventional pathway families. The unoriented network diagram and its oriented form are shown in Fig. \ref{fig:network}(a) and (b) respectively. Pathway edges are represented in black. There are six mathematical pathways in the network and some can be collectively grouped into four biological pathways. To satisfy our mathematical definition, a single node (representing the protein PTEN) which acts as an intermediary between the biological pathways $P_3$ and $P_4$ is technically labelled as its own pathway. The red and blue colours indicate activation and inhibition cross-talks and/or external stimulus, respectively. The diagram includes various coupled biological pathways, such as Wnt/$\beta$-catenin ($P_1$), EGFR family ($P_2$), Notch family ($P_3$), and TNF-R pathway ($P_4$). The two output nodes labelled $x_1$ and $x_2$ are shown in orange and we do not demand similarity in any of the other nodes. We choose these two nodes as the output as this model has been constructed to focus on the EGFR pathway and how it is affected by the other signalling pathways \cite{yamaguchi2014signaling}.

Immediately we highlight the power of the oriented form. In ensuring that all pathways are oriented, we can see the effect of the stimulii on this network. In the oriented form all stimulii except the stimulii of $P_4$ are contributing to activation of the output via the direct pathways. Furthermore, it is more easy to see the effect of the cross talk interactions. All cross-talk in the oriented form is positively stimulatory. That is, they all contribute to positive feedback or feed forward through cross-talk. The exception, of course, is the cross-talk with $P_4$ which is inhibited by its stimulus and through single inhibitory cross-talk with other pathways also acts as a positive stimulant. It is satisfying therefore that a network which is actively stimulated and contains positive feedback behaves (despite its added complexity) as a simple pathway in regards to the bias in the model. We can see this explicitly in the plots of $\delta_{ss}$ in Fig. \ref{fig:dss_network_twoway}. In particular, we notice that error $\delta_{ss}$ is greatly restricted to negative bias which has been an observation derived from our simpler tests for activated oriented pathways with positive feedback.

\begin{figure}[H]
\centering
\includegraphics[width=1\textwidth]{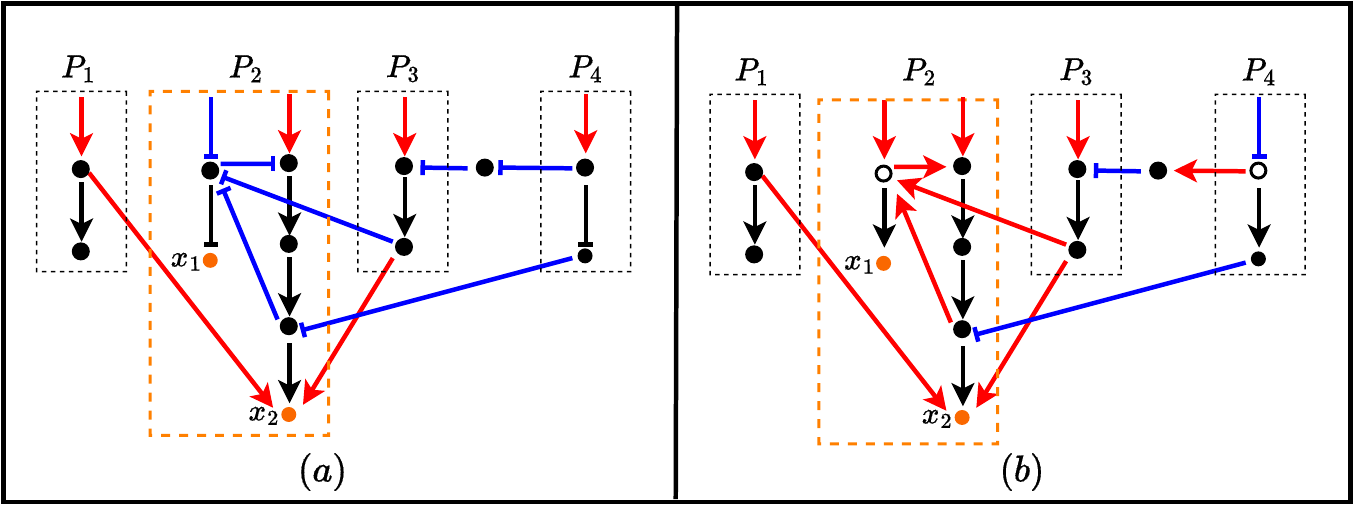}
\caption{A network model of EGFR/HER2 and its integration with other pathways as described in \cite{yamaguchi2014signaling} in both its unoriented form (a) and oriented form (b). The dashed boxes represent the four underlying interacting biological pathways. These are $P_1$ Wnt/$\beta$-catenin, $P_2$ EGFR, $P_3$ Notch and $P_4$ TNF-R. On the other hand, individual columns (shown with black edges) represent the six pathways as we define them mathematically in this manuscript. The red and blue arrows indicate activation and inhibition cross-talks and stimulii, respectively. We focus on the outputs of the model in \cite{yamaguchi2014signaling} which are indicated in orange and labelled $x_1$ and $x_2$.}
\label{fig:network}
\end{figure}

\begin{figure}[H]
\centering
\begin{subfigure}{.47\textwidth}
    \centering
    \includegraphics[width=.95\linewidth]{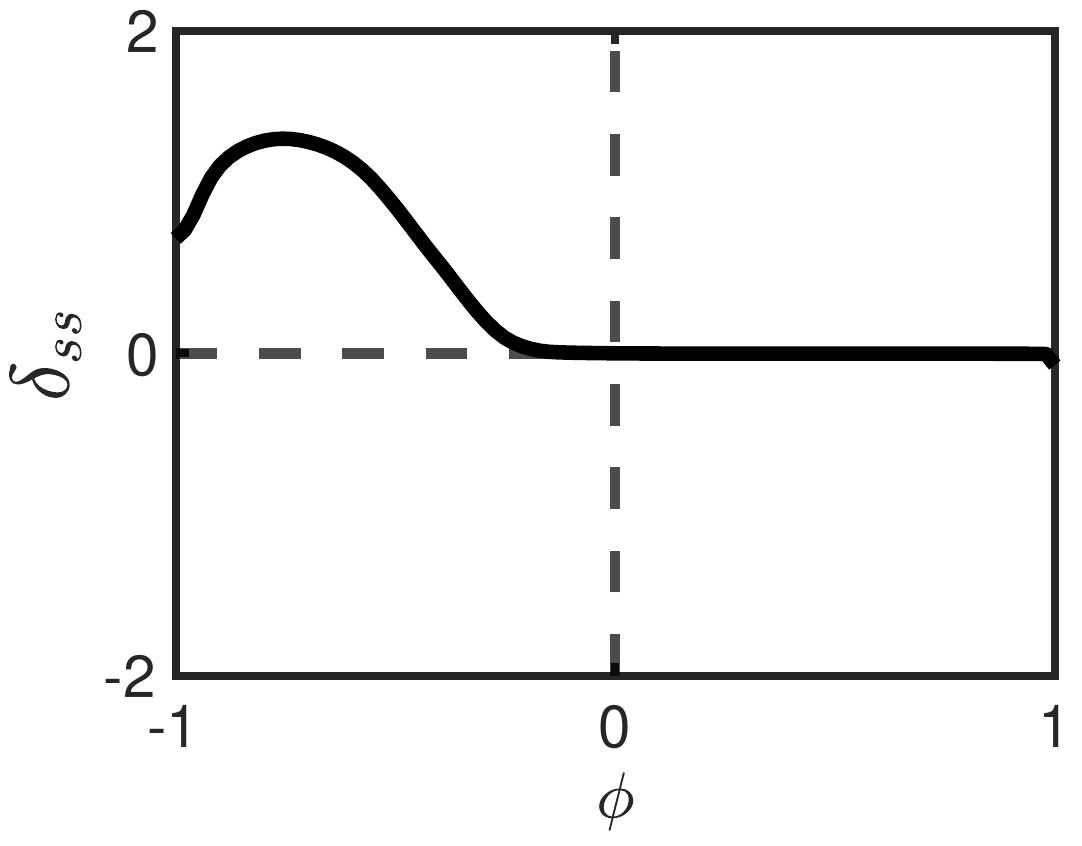}  
    \caption{}
    \label{fig:dss_x4_twoway}
\end{subfigure}
\begin{subfigure}{.47\textwidth}
    \centering
    \includegraphics[width=.95\linewidth]{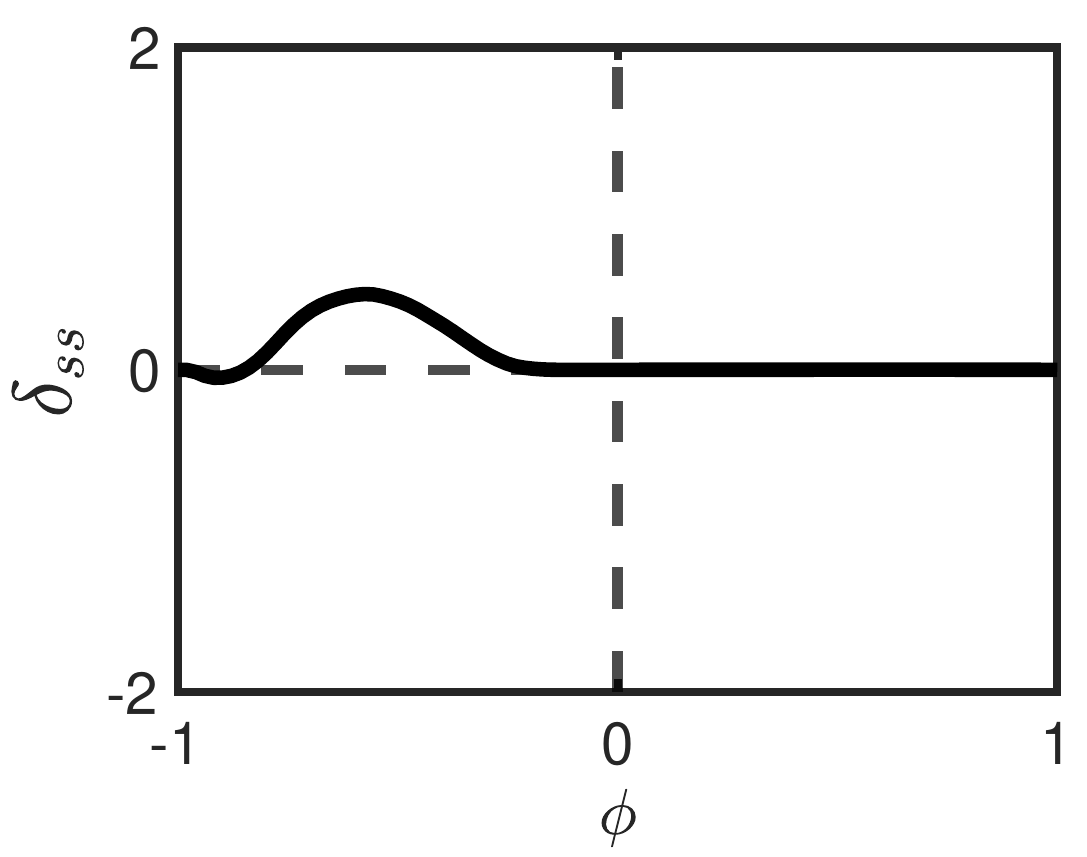}  
    \caption{}
    \label{fig:dss_x8_twoway}
\end{subfigure}
\caption{Steady state error $(\delta_{ss})$ for output nodes $x_1$ and $x_2$ using the test model described in Section \ref{testmodel}. The steady-state error accumulates as $\phi < 0$, but both networks maintain near equivalence when $\phi > 0$.}
\label{fig:dss_network_twoway}
\end{figure}

We can also visualise (more explicitly than simply plotting $\delta_\tau$ for this system) the discrepancy in the temporal evolution of the nodes $x_1$ and $x_2$ by plotting $x_1(t)$ and $x_2(t)$ explicitly for $\phi = 0$ and $\phi = \pm 0.5$ (representative of positive and negative bias respectively). This is done in the set of charts in Fig. \ref{fig:x4_and_x8}.  The first row of the figure shows the non-oriented output of $x_1$ (purple), while the second row shows the oriented network output. The third and fourth rows display the non-oriented and oriented network outputs for $x_2$ (blue) respectively. We observe perfect agreement for the unbiased model (centre column) as expected and very good agreement for the positively biased model as a result of positive stimulus and positive feedback throughout the oriented network (right column). Furthermore, as expected, the negatively biased model exhibits significant discrepancies between the oriented and unoriented forms in line with the observations of $\delta_{ss}$ in Fig. \ref{fig:dss_network_twoway} and consistent with observations of simpler pathways.

\begin{figure}[H]
\centering
\includegraphics[width=0.8\textwidth]{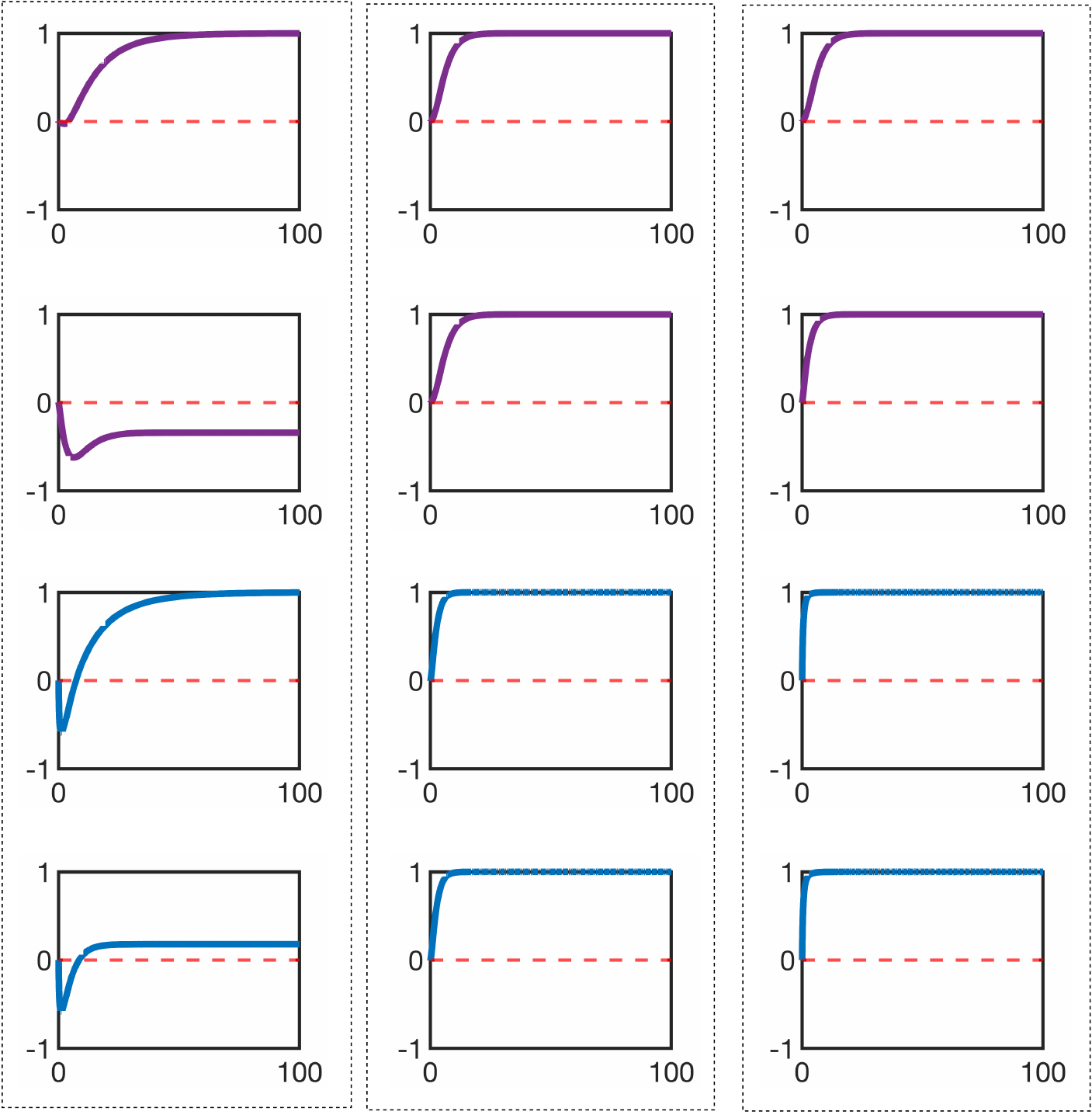}
\caption{Diagrams show the output of $x_1$(purple) and $x_2$(blue) in the model shown in Fig. \ref{fig:network}. The simulation output is averaged over many combinations of parameters $\alpha$ and $\beta$ for each edge and shown in three columns for $\phi =-0.5$ (left), $\phi = 0$ (centre), and $\phi = 0.5$ (right). The first and third rows display the non-oriented output of $x_1$ and $x_2$, while the second and the fourth row show the respective oriented network outputs.}
\label{fig:x4_and_x8}
\end{figure}

%--------------------------------------------------------------------------------------------
%---------------------------------- END OF Results---------------------------------------------
%--------------------------------------------------------------------------------------------

%--------------------------------------------------------------------------------------------
%---------------------------------- Conclusions  ---------------------------------------------
%--------------------------------------------------------------------------------------------
\section{Discussion and conclusion}\label{sec:conclusion}

This manuscript is concerned with the appropriateness of reducing the complexity of a biochemical network/pathway by assuming that inhibition and activation relationships behave as perfect inverses of each other. Such reduction is common as a method of identification of topological properties and possible qualitative behaviours but this reduction is only appropriate under certain conditions explored in this manuscript.

We define an oriented form as a standard/simplest way of reducing a pathway/network. We also prove that if a model is so-called `unbiased' and `symmetric' then reduction to the oriented form by successive exchanging of activations and inhibitions (keeping track of all operations by `flipping' the sign of the relevant nodes) produces a model indistinguishable (from the perspective of measured output) from the unoriented form.

We focus our study on the nature of (positive and negative) bias in the formation of error between an unoriented form and its reduced oriented form (both in the case of differing steady states and differing dynamic properties). We restrict ourself to symmetric models only. Our investigation is computational. Pathways in their oriented form have more easily identifiable structure and our investigation leads us to propose the following general conjectures for symmetrically modelled pathways and networks. 
\begin{enumerate}
\item For many systems bias can be added to a model without generating significant errors in the orienting process. When errors appear they do so rapidly and suddenly and therefore care should be exercised when reducing a network qualitatively in this way.
\item When the oriented network is \textit{externally stimulated} then the reduction to the oriented form produces significant errors if the model is negatively biased. If it is \textit{externally inhibited} errors are formed if the model is positively biased.
\item Errors are compounded as more approximations (replacing activators for inhibitors and flipping a node to compensate) are taken. Furthermore, if these approximations are made further downstream, the errors are expected to be larger.
\item The former statements extend to pathways with feedback. However, error in the steady state is observed for both negative and positive bias models if the feedback is negative (and the loop includes a flipped node) and error in the temporal behaviour is observed for both negative and positive bias models if the feedback is positive. 
\item The conclusions seem to be consistent if these general properties can be identified in the oriented form of more complex network models.
\end{enumerate}

This study has many limitations and suggests very strongly more rigorous work that should be done. We have opted not to look into detail at symmetry in this paper. This is partly due to the length of the paper but also because bias seems more common/significant in mathematical models in the literature. Furthermore, we do not have an objective definition to measure the extent of bias that makes sense. We have a measure of bias for the model used in this paper $\phi$ but this measure is arbitrary. It is therefore important that the scale of $\phi$ not be given too much emphasis. We also investigate one single model. We attempted to create this model with symmetry but also with the kinds of nonlinear relationships common in biochemical network models. It remains as future work to investigate the generalisability and analysis of the conjectures posed in this manuscript to a much more broad class of model. 

The observations found in this study form a framework with which to assess biochemical networks and determine if qualitative reductions are appropriate or if typologies are likely to be more idiosyncratic based on the specific quantitative model used to simulate it.

\backmatter

\bibliography{sn-bibliography}
\end{document}